\documentclass[12pt,amssymb,amscd]{amsart}
\usepackage{graphicx}
\usepackage{amsthm,amsmath,amssymb}
\usepackage{mathrsfs}
\usepackage{geometry}
\usepackage{setspace}
\usepackage{amsfonts}
\usepackage{tikz-cd}
\usepackage{color}
\newtheorem{theo}{Theorem}[section]
\newtheorem{prop}{Proposition}[section]
\newtheorem{lem}{Lemma}[section]

\theoremstyle{definition}
\newtheorem{defi}{Definition}[section]
\theoremstyle{remark}
\newtheorem{rem}{Remark}[section]
\date{}

\begin{document}

\relax

\title{$\mathfrak{g}_{0}$-Geometry and Its Relationship with Stable Representations}
\author[]{Shang Xu}
\date{ \today }
	\address{Dept. of Mathematics, Nankai University, TianJin, 300071, China}
	\email{shangxu086@gmail.com}
\maketitle
\begin{abstract}
\noindent This paper explores the sheaves with the action of a lie algebra and computes their cohomology in a new category. Then in the following sections, We try to generalize a classical result in [GM, Ch. IV] about exterior algebra. We add the action of $\mathfrak{g}_{0}$ and prove that there is still a faithful functor.
\end{abstract}

\tableofcontents

\section{Introduction}
In [GM, Ch.IV, \S 3], a fully faithful functor $\Phi$ is constructed for the Grassmannian algebra $\Lambda E$ from its stable module category to the derived category of coherent sheaves on projective plane. It's an easy observation that $\Lambda E$ is the universal enveloping algebra for the trivial lie superalgebra $\mathfrak{g}$ satisfying that $\mathfrak{g}_{0}=0,\,\mathfrak{g}_{1}=E$ with projective self-commuting $\mathcal{P}(E)$, so it's natural to consider the generalization of this result to a lie superalgebra. In order to obtain complexes of sheaves and to generalize the graded module category for Grassmannian algebra, we consider the lie superalgebra with $[\mathfrak{g}_{1},\mathfrak{g}_{1}]=0$. Furthermore, we need $\mathfrak{g}_{0}$ to be semisimple to apply results in \cite{DS}.

Let we conclude as follows: the aim of this paper is to simulate the method in that book and generalize the conclusion to a lie superalgebra $\mathfrak{g}=\mathfrak{g}_{0}\oplus\mathfrak{g}_{1}$ with $[\mathfrak{g}_{1},\mathfrak{g}_{1}]=0$ and $\mathfrak{g}_{0}$ semisimple. More explicitly, we will generalize the graded $\Lambda$-module category to the graded $\mathfrak{g}$-module category and find some geometric interpretation of this new category, namely the $\mathfrak{g}_{0}$-rigid complexes of coherent sheaves on $\mathcal{P}(\mathfrak{g}_{1})$ in section 4.2. The key point here is to introduce $\mathfrak{g}_{0}$ action on sheaves and introduce a new structure sheaf $\mathscr{R}_{X}$ to understand the category of sheaves with $\mathfrak{g}_{0}$-action, namely, the category of $\mathscr{R}_{X}$-modules. This part is particularly standard and gives us a general way to deal with the lie algebra action on sheaves, the significance of which may go beyond our original question. Therefore, it worth a whole chapter to describe what we have got in "$\mathfrak{g}_{0}$-algebraic geometry", and we put everything in section 2.2 and section 3. We also compute sheaf cohomology and other the $\mathrm{Ext}$ functor in the category of sheaves with $\mathfrak{g}_{0}$-action. To some extent, section 3.2 is the rewriting of [HR, Ch.III] for our structure sheaf $\mathscr{R}_{X}$.

Finally, in section 6 we will consider the more general case and reveal the relationship between our complex and the localization of DS functor [DS, \S 11]. Now in the next section, we will make some preparations for section 3 to introduce the $\mathfrak{g}_{0}$ geometry.

\textbf{Acknowledgement} I extend my heartfelt gratitude to Professor Vera Serganova for her invaluable inspiration and patient guidance, which were instrumental in shaping the main result of this paper.

\section{Definitions and Basic Properties}

\quad Throughout this paper we assume that $ \mathfrak{g}=\mathfrak{g}_{0}\oplus\mathfrak{g}_{1}$ is a finite dimensional lie superalgebra over field $k$, $\mathrm{char}\,k\neq 2$.

\subsection{Self-commuting cone and associated variety}

Some geometric notions of lie superalgebra $\mathfrak{g}$ arise in [DS, \S 11], we introduce their projectivizations here.

\begin{defi}
For a superalgebra $\mathfrak{g}$, the projective self-commuting cone $X$ is defined by
$$X=\{\bar{x}\in\mathcal{P}(\mathfrak{g}_{1})|x\in\mathfrak{g}_{1},[x,x]=0\}
$$
in which $\mathcal{P}(\mathfrak{g}_{1})$ is the projectivization of of $\mathfrak{g}_{1}$.
\end{defi}

For superalgebra $\mathfrak{g}$, its projective self-commuting cone is just $\mathcal{P}(\mathfrak{g}_{1})$. Now let M be a $\mathfrak{g}$-module. For every $x\in \mathfrak{g}_{1}$, the corresponding element $x_{M}\in End_{k}(M)$ satisfies $x_{M}^{2}=0$. Set
$$M_{x}:=\mathrm{Ker}\, x_{M}/xM
$$

It's clear that $M_{x}=M_{cx}$ for $c\in k$, so $M_{x}$ is actually determined by the projective image of x in $\mathcal{P}(\mathfrak{g}_{1})$. Hence we have a well-defined concept:

\begin{defi}
For a $\mathfrak{g}$-module M, we define the projective associated variety of M to be the set
$$X_{M}:=\{\bar{x}\in X|M_{x}\neq 0\}
$$
\end{defi}

One can prove that $X_{M}$ is a Zariski closed subvariety of $X$. Indeed, since $x^{2}_{M}=0$, the Jordan canonical form of $x_{M}$ has a collection of $2\times2$ and $1\times1$ Jordan blocks on the diagonal. Hence $M_{x}=0$ is equivalent to x has only $2\times2$ Jordan blocks on the diagonal of the canonical form, which is the same as the rank of $x_{M}$ equals, or equivalently, greater than $\frac{1}{2}dim\,M$. In other words, 
$$X_{M}=\{\bar{x}\in X|\,\mathrm{rank}\,x_{M}<\frac{1}{2}dim\,M\}
$$
which is Zariski closed in $X$.

In the following contexts, we always identify these varieties with their associated schemes.

Now $\mathfrak{g}_{0}$ has a natural action on sections of structure sheaf $\mathscr{O}_{X}$. Indeed, suppose $e_{1},e_{2},...,e_{n}$ is a basis for $\mathfrak{g}_{0}$, then for $x\in \mathfrak{g}_{0}$ an open set $U\subseteq X$, $f\in\mathscr{O}_{X}(U)$, if $f=e_{i}^{*}$, we define the action by
$$(x.e_{i}^{*})(y)=e_{i}^{*}([x,y]), \quad y\in \mathfrak{g}_{1}
$$
And the action of x on polynomials in $e_{1}^{*},e_{2}^{*},...,e_{n}^{*}$(=regular functions on U) is induced by linearity and Leibniz rule:
$$x.(fg)=x.(f)g+fx.(g)
$$
One can check that this action is compatible with restrictions $\mathscr{O}_{X}(U)\rightarrow\Gamma(V,\mathscr{O}_{X})$ for $V\subseteq U$ and turns rings of sections $\mathscr{O}_{X}(U)$ into $\mathfrak{g}_{0}$-modules.

\subsection{Sheaves with $\mathfrak{g}_{0}$-action}

The action of $\mathfrak{g}_{0}$ on the self-commuting cone inspires us to consider the sheaves of $\mathscr{O}_{X}$-module which have actions of $\mathfrak{g}_{0}$ on all sections.

\begin{defi}
Suppose $X$ is a scheme over k with structure sheaf $\mathscr{O}_{X}$. Lie algebra $\mathfrak{g}_{0}$ has an action on sections of $\mathscr{O}_{X}$ that is compatible with restrictions. A sheaf of $\mathfrak{g}_{0}\mathscr{O}_{X}$-module $\mathscr{M}$ is a sheaf of $\mathscr{O}_{X}$-modules with $\mathfrak{g}_{0}$ action (as a lie algebra) on sections by Leibniz rule: for an open set $U\subseteq X$ and $x\in\mathfrak{g}_{0}, r\in\mathscr{O}_{X}(U),m\in\mathscr{M}(U)$
$$x(rm)=x(r)m+r(xm)
$$
Furthermore, we require x to be an endomorphism of sheaf $\mathscr{M}$, in other words, for open set $V\subseteq U$, we have 
$$x(m|_{V})=(xm)|_{V}
$$
\end{defi}

We can define a category $\mathfrak{g}_{0}\mathfrak{Mod}(X)$ with objects $\mathfrak{g}_{0}\mathscr{O}_{X}$-module. Morphisms between two $\mathfrak{g}_{0}\mathscr{O}_{X}$-modules $\mathscr{M}$ and $\mathscr{N}$ are morphisms of $\mathscr{O}_{X}$-modules which are compatible with $\mathfrak{g}_{0}$ action, in other words, the following diagram commute
\begin{center}
\begin{tikzcd}
\mathscr{M} \arrow[r, "f"] \arrow[d,"x"]
&\mathscr{N} \arrow[d, "x"] \\
\mathscr{M} \arrow[r, "f" ]
& \mathscr{N}
\end{tikzcd}
\end{center}

In the next section we will construct a sheaf $\mathscr{R}_{X}$ and prove that the category $\mathfrak{g}_{0}\mathfrak{Mod}(X)$ is actually the sheaves of $\mathscr{R}_{X}$-module, so that our category $\mathfrak{g}_{0}\mathfrak{Mod}(X)$ will have some good natures. For example, it is an abelian category and it has enough injectives.

For two objects $\mathscr{F},\,\mathscr{G}$ in $\mathfrak{g}_{0}\mathfrak{Mod}(X)$, define their tensor product over $k$ to be the sheafification of the presheaf $U\mapsto \mathscr{F}(U)\otimes_{k}\mathscr{G}(U)$, and the action of $\mathfrak{g}_{0}$ is induced by the action of $\mathfrak{g}_{0}$ on this presheaf: for $m\in\mathscr{F}(U),\,n\in\mathscr{G}(U)$ and $x\in\mathfrak{g}_{0}$, define
$$x(m\otimes n)=xm\otimes n+m\otimes xn
$$

A sheaf of $\mathfrak{g}_{0}\mathscr{O}_{X}$-module is said to be quasi-coherent if it is also a quasi-coherent sheaf as an $\mathscr{O}_{X}$-module. Denote $\mathfrak{g}_{0}\mathfrak{Qco}(X)$\, to be the full subcategory of $\mathfrak{g}_{0}\mathfrak{Mod}(X)$ whose objects are quasi-coherent. As for the notion of coherent sheaves in $\mathfrak{g}_{0}\mathfrak{Mod}$, things are a little more complicated since we need a locally-finite condition for simultaneous action of $\mathscr{O}_{X}$ and $\mathfrak{g}_{0}$, which we will introduce in the next section.

\section{The structure sheaf and cohomology in $\mathfrak{g}_{0}\mathfrak{Mod}(X)$}

\subsection{The structure sheaf $\mathscr{R}_{X}$}

\quad In this section, suppose X is a scheme over k with $\mathfrak{g}_{0}$ action on sections of $\mathscr{O}_{X}$ that is compatible with restrictions.

For a sheaf of $\mathfrak{g}_{0}\mathscr{O}_{X}$-module $\mathscr{M}$ and open set $U\subseteq X$, $\mathscr{M}(U) $ will be both an $\mathscr{O}_{X}(U)$-module and an $\mathfrak{g}_{0}$-module. Denote \,$\mathcal{U}(\mathfrak{g}_{0})$ to be the universal enveloping algebra of $\mathfrak{g}_{0}$, then the $\mathfrak{g}_{0}$-module structure of $\mathscr{M}(U)$ is equivalent to a $\mathcal{U}(\mathfrak{g}_{0})$-module structure. Now consider the ring $R(U)$ defined by
$$R(U):=(\mathscr{O}_{X}(U)\ast\mathcal{U}(\mathfrak{g}_{0}))/I(U)
$$
in which $\ast$ is the free product of two associative k-algebras $\mathscr{O}_{X}(U)$ and $\mathcal{U}(\mathfrak{g}_{0})$ and $I(U)$ is the two-sided ideal generated by elements of the form $x\ast r-x(r)-r\ast x$ for  all $x\in\mathfrak{g}_{0}$, $r\in\mathscr{O}_{X}(U)$. (Here we identify $\mathfrak{g}_{0}$ as a subset of $\mathcal{U}(\mathfrak{g}_{0})$ by PBW theorem). One can see the module of $\mathscr{O}_{X}(U)$ and $\mathcal{U}(\mathfrak{g}_{0})$ structure that satisfies Leibniz rule is equivalent to an $R(U)$-module structure.

Now $R(U)$ is somewhat an enveloping algebra for the two module structure. It's ring structure is somewhat sophisticated, but it looks simpler as a $\mathscr{O}_{X}(U)$-module.

\begin{lem}
R(U) is a free $\mathscr{O}_{X}(U)$-module, in other words, 
$$R(U)\cong\mathscr{O}_{X}(U)\otimes_{k}\mathcal{U}(\mathfrak{g}_{0})
$$
as $\mathscr{O}_{X}(U)$ modules.
\end{lem}
\begin{proof}
We construct a lie algebra $\mathfrak{l}$ , the underlying vector space is $\mathfrak{g}_{0}\oplus\mathscr{O}_{X}(U)$. We define the lie bracket by
\begin{equation}\notag
  [x,y]_{\mathfrak{l}}=\begin{cases}
    [x,y]_{\mathfrak{g}_{0}} & \text{if $x,y\in\mathfrak{g}_{0}$}\\
    x(y) & \text{if 
$x\in\mathfrak{g}_{0}, y\in\mathscr{O}_{X}(U)$}\\
-y(x) & \text{if 
$x\in\mathscr{O}_{X}(U),y\in\mathfrak{g}_{0}$}\\
0  & \text{if $x,y\in\mathscr{O}_{X}(U)$}
  \end{cases}
\end{equation}
It's straightforward to check that this bracket satisfies the lie algebra axioms. Denote $\mathcal{U}(\mathfrak{l})$ to be the universal enveloping algebra of $\mathfrak{l}$. The universal enveloping algebra of the trivial lie subalgebra $\mathscr{O}_{X}(U)$ in $\mathfrak{l}$ is the symmetric algebra $\mathcal{S}(\mathscr{O}_{X}(U))$ (corresponding to the basis of monomials with lexicographic order). By PBW Theorem,
$$\mathcal{U}(\mathfrak{l})\cong\mathcal{S}(\mathscr{O}_{X}(U))\otimes_{k}\,\mathcal{U}(\mathfrak{g}_{0})
$$
as $\mathcal{S}(\mathscr{O}_{X}(U))$-modules. Now denote $J(U)$ to be the ideal in $\mathcal{S}(\mathscr{O}_{X}(U))$ generated by $f\otimes g-fg,\, f,g\in\mathscr{O}_{X}(U)$. Then $\mathcal{S}(\mathscr{O}_{X}(U))/J(U)\cong\mathscr{O}_{X}(U)$ as rings. Then there is an isomorphism of $\mathscr{O}_{X}(U)$(= $\mathcal{S}(\mathscr{O}_{X}(U))/J(U)$)-modules
$$
\mathcal{U}(\mathfrak{l})/J(U)\mathcal{U}(\mathfrak{l})\cong\mathcal{S}(\mathscr{O}_{X}(U))/J(U)\otimes_{k}\mathcal{U}(\mathfrak{g}_{0})=\mathscr{O}_{X}(U)\otimes_{k}\mathcal{U}(\mathfrak{g}_{0})
$$
On the other hand, by definition,
$$\mathcal{U}(\mathfrak{l})=(\mathcal{S}(\mathscr{O}_{X}(U))\ast\mathcal{U}(\mathfrak{g}_{0}))/I'(U)
$$
where $I'(U)$ is the two sided ideal in $\mathcal{S}(\mathscr{O}_{X}(U))\ast\mathcal{U}(\mathfrak{g}_{0})$ generated by all elements of the form $x\ast r-x(r)-r\ast x$, $x\in\mathfrak{g}_{0}$, $r\in\mathscr{O}_{X}(U)$. Now let $J'(U)$ be the two-sided ideal in $\mathcal{S}(\mathscr{O}_{X}(U))\ast\mathcal{U}(\mathfrak{g}_{0})$ generated by generated by $f\otimes g-fg$. One can see that:
\begin{align}
\notag
\mathcal{U}(\mathfrak{l})/J(U)\mathcal{U}(\mathfrak{l})
&\cong(\mathcal{S}(\mathscr{O}_{X}(U))\ast\mathcal{U}(\mathfrak{g}_{0}))/(I'(U)+J'(U))\\\notag
&\cong(\mathscr{O}_{X}(U)\ast\mathcal{U}(\mathfrak{g}_{0}))/((I'(U)+J'(U))/J'(U))\\\notag
&\cong((\mathscr{O}_{X}(U)\ast\mathcal{U}(\mathfrak{g}_{0}))/I(U)\\\notag
&=R(U)
\end{align}
as $\mathscr{O}_{X}(U)$-modules. Hence we have an $\mathscr{O}_{X}(U)$-module isomorphism
$$R(U)\cong\mathcal{U}(\mathfrak{l})/J(U)\mathcal{U}(\mathfrak{l})\cong\mathscr{O}_{X}(U)\otimes_{k}\mathcal{U}(\mathfrak{g}_{0})
$$
\end{proof}

Notice that the restriction $\mathscr{O}_{X}(U)\rightarrow\mathscr{O}_{X}(V)$ naturally induce a ring homomorphism $R(U)\rightarrow R(V)$. This allows us to construct a presheaf
$$U\mapsto R(U)
$$
By Lemma 1, this presheaf is actually a sheaf that is isomorphic to $\mathcal{U}(\mathfrak{g}_{0})\otimes_{k}\mathscr{O}_{X}$ as a sheaf of $\mathscr{O}_{X}$-modules. We name this sheaf $\mathscr{R}_{X}$. Indeed, $\mathscr{R}_{X}$ plays a similar role in $\mathfrak{g}_{0}\mathfrak{Mod}(X)$ as the structure sheaf $\mathscr{O}_{X}$ in $\mathfrak{Mod}(X)$: Firstly, the $\mathfrak{g}_{0}$ action on sections of $\mathscr{M}$ that satisfies Leibniz rule corresponds to the $R(U)$-module structure on $\mathscr{M}(U)$, as we discussed in the beginning of this section. Secondly, the compatibility of  $\mathfrak{g}_{0}$ action with restrictions is the same as the compatibility of module structure via the ring homomorphism $R(U)\rightarrow R(V)$. Finally, the morphisms in $\mathfrak{g}_{0}\mathfrak{Mod}(X)$ are actually morphisms of $\mathscr{R}_{X}$-modules. Therefore, $\mathfrak{g}_{0}\mathfrak{Mod}(X)$ is a category of $\mathscr{R}_{X}$-modules, and we call $\mathscr{R}_{X}$ by $\mathfrak{g}_{0}$-structure sheaf or simply structure sheaf if there is no confusion.

\begin{rem}
The $\mathfrak{g}_{0}$ action on the system $\mathscr{O}_{U}$ induces an action on stalks $\mathscr{O}_{x,X}$, which is the same as the $\mathscr{R}_{x,X}$($\cong(\mathscr{O}_{x,X}\ast\mathcal{U}(\mathfrak{g}_{0}))/I_{x}$)-module structure. 
\end{rem}
\begin{rem}
One can easily check that a sheaf $\mathscr{R}_{X}$ in $\mathfrak{g}_{0}\mathfrak{Mod}(X)$ is quasi-coherent if and only if there is a open covering $X=\cup_{i}U_{i}=\cup_{i}SpecA_{i}$ such that for all $U_{i}$ (or equivalently, for all open affine subsets, as [HR, Ch. II, 5.4] shows), the restriction of $\mathscr{M}$, $\mathscr{M}|_{U_{i}}$, is of the form $\tilde{M_{i}}$ for some $\mathscr{R}(U_{i})$ ($=(A_{i}\ast\mathcal{U}_{\mathfrak{g}_{0}})/I(U_{i})$)-module M. Equivalently, if for any point $x\in X$, there is a open neighberhood $U$ of X and an exact sequence
$$\mathscr{R}_{U}^{\oplus I}\rightarrow\mathscr{R}_{U}^{\oplus J}\rightarrow\mathscr{M}\rightarrow 0
$$
for some sets $I$ and $J$. However, sometimes it's more convenient to use the definition that a $\mathfrak{g}_{0}$ quasi-coherent sheaf is a quasi-coherent $\mathscr{O}_{X}$ module with $\mathfrak{g}_{0}$ action. Especially, when we cite results in [HR, Ch.II, \S 5], we always firstly think of the case of $\mathscr{O}_{X}$-module and then add $\mathfrak{g}_{0}$-action.
\end{rem}

Now with the notion of structure sheaf $\mathscr{R}_{X}$ and the remark above, we can give a definition of coherent sheaves in $\mathfrak{g}_{0}\mathfrak{Mod}(X)$.
\begin{defi}
A quasi-coherent sheaf of $\mathscr{R}_{X}$-modules $\mathscr{M}$ is said to be coherent if for a open covering $X=\cup_{i}U_{i}=\cup_{i}SpecA_{i}$ with $\mathscr{M}|_{U_{i}}=\tilde{M_{i}}$, $M_{i}$ is finitely generated as an $\mathscr{R}_{X}(U_{i})$-module. Equivalently, if for any point $x\in X$, there is a open neighberhood $U$ of X and an exact sequence
$$\mathscr{R}_{U}^{\oplus I}\rightarrow\mathscr{R}_{U}^{\oplus J}\rightarrow\mathscr{M}\rightarrow 0
$$
for set $I$ and finite set $J$
\end{defi}

Similarly, we can define the locally free objects in $\mathfrak{g}_{0}\mathfrak{Mod}(X)$ to be the sheaf of $\mathscr{R}_{X}$ module which admits an open covering $\{U_{i}\}$ of X  such that in each $U_{i}$, it is isomorphic to a direct sum of copies of $\mathscr{R}_{X}|_{U_{i}}$.

\subsection{Derived functors in $\mathfrak{g}_{0}\mathfrak{Mod}(X)$ and $\mathfrak{g}_{0}\mathfrak{Qco}(X)$}

Next, we exploit some basic nature of cohomology in category $\mathfrak{g}_{0}\mathfrak{Mod}(X)$. Firstly, as in the category $\mathfrak{Mod}(X)$, $\mathfrak{g}_{0}\mathfrak{Mod}(X)$ has enough injectives.

\begin{prop}
Any object in $\mathfrak{g}_{0}\mathfrak{Mod}(X)$ admits a injection into an injective objects, in other words, the category $\mathfrak{g}_{0}\mathfrak{Mod}(X)$ has enough injectives.
\end{prop}
\begin{proof}
$(X,\mathscr{R}_{X})$ is a ringed space, so this result follows from [HR, Ch.III, 2.2]
\end{proof}

Now with this proposition, we can define the $\mathfrak{g}_{0}$-sheaf cohomology $H_{\mathfrak{g}_{0}}^{i}(X,\cdot)$ to be the derived functor of global section functor $\Gamma(X,\cdot)$ from  $\mathfrak{g}_{0}\mathfrak{Mod}(X)$ to $\mathfrak{Ab}$.  Recall that a sheaf $\mathscr{F}$ is \emph{flasque} if the restriction maps $\mathscr{F}(U)\rightarrow\mathscr{F}(V)$ are surjective. Again since $(X,\mathscr{R}_{X})$ is a ringed space, by [HR, Ch.III, 2.3; 2.4; 2.6], We have the following result.
\begin{prop}
Injective $\mathscr{R}_{X}$-modules are flasque and therefore  acyclic for the functor $\Gamma(X,\cdot)$. Hence, the derived functors of the global section functor from $\mathfrak{g}_{0}\mathfrak{Mod}(X)$ to $\mathfrak{Ab}$ coincide with the sheaf cohomology functor $H^{i}(X,\cdot)$ from $\mathfrak{Ab}(x)$ to $\mathfrak{Ab}$.
\end{prop}

Now we focus our attention on the cohomology in $\mathfrak{g}_{0}\mathfrak{Qco}(X)$. Firstly, this category also have enough injectives when $X$ is noetherian, as the following proposition shows.
\begin{prop}
Suppose $X$ is noetherian. Let $\mathscr{M}$ be a quasi-coherent sheaf of $\mathscr{R}_{X}$-modules, then $\mathscr{M}$ can be embedded into a quasi-coherent sheaf that is injective as an $\mathscr{R}_{X}$-module. Hence the category $\mathfrak{g}_{0}\mathfrak{Qco}(X)$ has enough injectives.
\end{prop}
\begin{proof}
Suppose $\{U_{i}\}_{i\in I}$ is a family of affine open sets that covers X. Now as in [HR, Ch.II, \S 5],  
since $U_{i}$'s are affine, there are equivalences of categories: $\mathfrak{g}_{0}\mathfrak{Qco}(U_{i})\cong\mathfrak{Mod}(\mathscr{R}_{X}(U_{i}))$. For each $i\in I$, suppose $U_{i}=\mathrm{Spec}A_{i}$ and $\mathscr{M}(U_{i})=\tilde{M_{i}}$, in which $M_{i}$ is an $\mathscr{R}_{X}(U_{i})$-module. Now there is an injection $M_{i}\rightarrow I_{i}$, in which $I_{i}$ is injective as an $\mathscr{R}_{X}(U_{i})$-module, so correspondingly, $\tilde{M_{i}}$ can be embedded into $\tilde{I_{i}}$ that is an injective object in $\mathfrak{g}_{0}\mathfrak{Qco}(U_{i})$.

Now suppose $f_{i}$ is the open immersion of schemes $\mathrm{Spec}A_{i}\rightarrow\mathscr{M}$. Since $X$ is noetherian, so is each $A_{i}$. Then by [HR, Ch.II, \S 5], both the inverse image functor ${f_{i}}^{\ast}$ and the direct image functor ${f_{i}}_{\ast}$ send quasi-coherent sheaves to quasi-coherent sheaves. Furthermore, there is an adjunction of functors:
$$ \mathrm{Hom}_{\mathscr{R}_{X}|_{U_{i}}}({f_{i}}^{\ast}\mathscr{G}, \mathscr{F})\cong\mathrm{Hom}_{\mathscr{R}_{X}}(\mathscr{G}, {f_{i}}_{\ast}\mathscr{F})
$$
On the other hand, ${f_{i}}^{\ast}$ is exact since $f_{i}$ is flat. Therefore, ${f_{i}}_{\ast}$ is right adjoint to an exact functor, by [WC, Ch.II, \S 3], ${f_{i}}^{\ast}$ preserves injectives. Hence ${f_{i}}_{\ast}(\tilde{I_{i}})$ is injective in $\mathfrak{g}_{0}\mathfrak{Qco}(X)$. Again by that adjunction there is a map $\mathscr{M}\rightarrow{f_{i}}_{\ast}(\tilde{I_{i}})$ that equals $M_{i}\rightarrow I_{i}$ when restricted to $U_{i}$, hence is injective at stalks in $U_{i}$. Now consider the product of maps $\mathscr{M}\rightarrow\prod_{i}{f_{i}}_{\ast}(\tilde{I_{i}})$, it is injective at stalks everywhere and therefore injective as map of sheaves. Furthermore, $\prod_{i}{f_{i}}_{\ast}(\tilde{I_{i}})$ is a product of injectives, hence is also injective (it is quasi-coherent obviously). Then $\mathscr{M}$ admits an injection into a injective object in $\mathfrak{g}_{0}\mathfrak{Qco}(X)$, completing the proof.

\end{proof}

This proposition allows us to define right derived functors in $\mathfrak{g}_{0}\mathfrak{Qco}(X)$. We define sheaf cohomology $H^{i}_{\mathfrak{g}_{0}\mathfrak{Qco}(X)}(X,\cdot)$ as the derived functor for global section functor $\Gamma(X,\cdot)$ and $\mathrm{Ext}^{i}_{\mathfrak{g}_{0}\mathfrak{Qco}(X)}(\mathscr{F},\cdot)$ as the derived functor for $\mathrm{Hom}_{\mathfrak{g}_{0}\mathfrak{Qco}(X)}(\mathscr{F},\cdot)$. We have two basic properties of them.
\begin{prop}
For any object $\mathscr{F}$ in $\mathrm{Ob}\,\mathfrak{g}_{0}\mathfrak{Qco}(X)$, we have

$\mathrm{(a)}$ $\mathrm{Ext}^{i}_{\mathfrak{g}_{0}\mathfrak{Qco}(X)}(\mathscr{R}_{X},\mathscr{F})\cong H^{i}_{\mathfrak{g}_{0}\mathfrak{Qco}(X)}(\mathscr{F}, X)$

$\mathrm{(b)}$ Let $\mathscr{L}$ be a locally free $\mathscr{R}_{X}$-module of finite rank, then we have 
$$\mathrm{Ext}^{i}_{\mathfrak{g}_{0}\mathfrak{Qco}(X)}(\mathscr{F}\otimes_{\mathscr{R}_{X}}\mathscr{L},\mathscr{G})\cong\mathrm{Ext}^{i}_{\mathfrak{g}_{0}\mathfrak{Qco}(X)}(\mathscr{F},\mathscr{G}\otimes_{\mathscr{R}_{X}}\mathscr{L}\,\check{\vrule height1.3ex width0pt}\,)
$$
\end{prop}
in which $\mathscr{L}\,\check{\vrule height1.3ex width0pt} =\mathscr{H}om_{\mathscr{R}_{X}}(\mathscr{L},\mathscr{R}_{X})$ is the dual of $\mathscr{L}$.
\begin{proof}
The functors $\mathrm{Hom}(\mathfrak{R}_{X},\cdot)$ and $\Gamma(X,\cdot)$ are identical, so their derived functors (as functors from $\mathfrak{g}_{0}\mathfrak{Qco}(X)$ to $\mathfrak{Ab}$) are the same. This proves $(a)$. 

Now if $\mathscr{I}\in\mathrm{Ob}\,\mathfrak{g}_{0}\mathfrak{Qco}(X)$ is injective, the functor $\mathrm{Hom}(\cdot,\mathscr{I})$ is exact. Now since $\mathscr{L}$ is locally free of finite rank, so is its dual $\mathscr{L}\,\check{\vrule height1.3ex width0pt}$. Then $\cdot\otimes_{\mathscr{R}_{X}}\mathscr{L}\,\check{\vrule height1.3ex width0pt}$ is also an exact functor. Hence $\mathrm{Hom}(\cdot\otimes_{\mathscr{R}_{X}}\mathscr{L}\,\check{\vrule height1.3ex width0pt},\mathscr{I})$ is also exact. By [HR, Ch. II, \S 5, Ex.1], we have 
$$\mathrm{Hom}_{\mathscr{R}_{X}}(\mathscr{F}\otimes_{\mathscr{R}_{X}}\mathscr{L}\,\check{\vrule height1.3ex width0pt},\mathscr{I})\cong\mathrm{Hom}_{\mathscr{R}_{X}}(\mathscr{F},\mathscr{L}\otimes_{\mathscr{R}_{X}}\mathscr{I})
$$
for any sheaf of $\mathscr{R}_{X}$ modules $\mathscr{F}$. Since $\mathfrak{g}_{0}\mathfrak{Qco}(X)$ is the full subcategory of $\mathfrak{g}_{0}\mathfrak{Mod}(X)$, these $\mathrm{Hom}$ sets can also be obtained in category $\mathfrak{g}_{0}\mathfrak{Qco}(X)$,
$$\mathrm{Hom}_{\mathfrak{g}_{0}\mathfrak{Qco}(X)}(\mathscr{F}\otimes_{\mathscr{R}_{X}}\mathscr{L}\,\check{\vrule height1.3ex width0pt},\mathscr{I})\cong\mathrm{Hom}_{\mathfrak{g}_{0}\mathfrak{Qco}(X)}(\mathscr{F},\mathscr{L}\otimes_{\mathscr{R}_{X}}\mathscr{I})
$$
Therefore the functor $\mathrm{Hom}_{\mathfrak{g}_{0}\mathfrak{Qco}(x)}(\cdot,\mathscr{L}\otimes_{\mathscr{R}_{X}}\mathscr{I})$ is also exact. This proves that $\mathscr{L}\otimes_{\mathscr{R}_{X}}\mathscr{I}$ is also injective. Similarly, tensoring with $\mathscr{L}\,\check{\vrule height1.3ex width0pt}$ also preserves injectives.

Now the case $i=0$ in $\mathrm{(b)}$ follows from the adjunction above. For general case, note that both sides are $\delta$-functors in $\mathscr{G}$ from $\mathfrak{g}_{0}\mathfrak{Qco}(X)$ to $\mathfrak{Ab}$, since tensoring with a locally free sheaf is exact. Now since tensoring with $\mathscr{L}\,\check{\vrule height1.3ex width0pt}$ preserves injectives, both sides are effaceable (When $\mathscr{G}$ is injective and $i\geq1$, both sides vanish). Then use [WC, Ch.II, \S 4, Ex.2.4.5] to conclude they are both universal, hence equal.
\end{proof}

The next striking result is that if $X$ is noetherian, $H^{i}_{\mathfrak{g}_{0}\mathfrak{Qco}(X)}(X,\cdot)$ coincides with the sheaf cohomology $H^{i}(X,\cdot)$ .

\begin{prop}
Suppose $X$ is noetherian, then any injective object in $\mathfrak{g}_{0}\mathfrak{Qco}(X)$ is flasque. Therefore, $H^{i}_{\mathfrak{g}_{0}\mathfrak{Qco}(X)}(X,\cdot)$ coincides with the sheaf cohomology $H^{i}(X,\cdot)$. 
\end{prop}

\begin{proof}
Notice that we have an adjunction of functors:
$$\mathrm{Hom}_{\mathfrak{g}_{0}\mathfrak{Qco}(X)}(\mathscr{R}_{X}\otimes_{\mathscr{O}_{X}}\mathscr{F},\mathscr{G})\cong\mathrm{Hom}_{\mathfrak{Qco}(X)}(\mathscr{F},\mathscr{G})
$$
for $\mathscr{O}_{X}$-module $\mathscr{F}$ and $\mathscr{R}_{X}$-module $\mathscr{G}$. By Lemma 1, $\mathscr{R}_{X}\cong\mathcal{U}(\mathfrak{g}_{0})\otimes_{k}\mathscr{O}_{X}$ as $\mathscr{O}_{X}$-modules, hence the tensoring functor $\mathscr{R}_{X}\otimes_{O_{X}}$ is an exact functor from $\mathfrak{g}_{0}\mathfrak{Qco}(X)$ to $\mathfrak{Qco}(X)$. Therefore, the forgetful functor from $\mathfrak{Qco}(X)$ to $\mathfrak{g}_{0}\mathfrak{Qco}(X)$ is right adjoint to an exact functor, hence preserves injectives by [WC, Ch.II, \S 3]. Then for an injective object $\mathscr{I}$ in $\mathfrak{g}_{0}\mathfrak{Qco}(X)$, it is also injective in $\mathfrak{Qco}(X)$. By [HR, Ch.III, \S 3, Ex.3.6 (b)], it is flasque, hence acyclic for global section functor from $\mathfrak{Ab}(X)$ to $\mathfrak{Ab}$. Then we can conclude that $H^{i}_{\mathfrak{g}_{0}\mathfrak{Qco}(X)}(X,\cdot)$ coincides with the sheaf cohomology $H^{i}(X,\cdot)$.
\end{proof}

Given these two propositions, we are able to make some explicit computation of the $\mathfrak{g}_{0}$-cohomology of projective space. Firstly, we clarify the definition of twisted $\mathscr{R}_{X}$-module.

\begin{defi}
Suppose the underlying scheme $X$ is isomorphic to the projective space $\mathbf{P}^{r}_{k}$. we define the twisting sheaves $\mathscr{R}_{X}(i)$ by
$$\mathscr{R}_{X}(i):=\mathscr{R}_{X}\otimes_{\mathscr{O}_{X}}\mathscr{O}_{X}(i) 
$$
\end{defi}

In general case, for a sheaf of $\mathscr{R}_{X}$ modules $\mathscr{M}$, we denote by $\mathscr{M}(n)$ the $\mathfrak{g}_{0}$-twisted sheaf $\mathscr{M}\otimes_{\mathscr{R}_{x}}\mathscr{R}_{X}(n)$
, which is isomorphic to the usual twisted $\mathscr{O}_{X}$-module since
$$\mathscr{M}\otimes_{\mathscr{R}_{x}}\mathscr{R}_{X}(n)\cong\mathscr{M}\otimes_{\mathscr{R}_{x}}\mathscr{R}_{X}\otimes_{\mathscr{O}_{X}}\mathscr{O}_{X}(i)\cong\mathscr{M}\otimes_{\mathscr{O}_{x}}\mathscr{O}_{X}(n)
$$
so there is no difference twisting by either of the structure sheaf. 
\begin{theo}
Suppose X is isomorphic to  $P_{k}^{r}(r\geq 1)$ as a scheme and $E$ the set of is its closed points, in other words, the projectivization of an $r+1$-dimensional vector space. Then:
\begin{equation}\notag
\mathrm{Ext}^{l}_{\mathfrak{g}_{0}\mathfrak{Qco}(X)}(\mathscr{R}_{X}(i),\mathscr{R}_{X}(j))
=\begin{cases}
\mathcal{U}(\mathfrak{g}_{0})\otimes_{k}S^{i-j}(E^{\ast}) & \text{for $l=0$}\\
\mathcal{U}(\mathfrak{g}_{0})\otimes_{k}S^{-n-1+i-j}(E^{\ast})\otimes_{k}\Lambda^{r+1}(E^{\ast}) & \text{for $l=n$}\\
0 & \text{otherwise}
\end{cases}
\end{equation}
where $S^{m}$ is the homogeneous symmetric tensor of order m, and $S^{m}(E^{\ast})=0$ if $m<0$. $\Lambda^{r+1}(E^{\ast})$ is the highest degree terms in the exterior algebra of $E^{\ast}$.
\end{theo}
\begin{proof}
Since $\mathscr{O}_{X}(n)$'s are locally free  $\mathscr{O}_{X}$-modules, we can see that $\mathscr{R}_{X}(n)$'s are locally free $\mathscr{R}_{X}$-modules. Furthermore, we have
$$
\mathscr{R}_{X}(n)\,\check{\vrule height1.3ex width0pt}=\mathscr{H}om_{\mathscr{R}_{X}}(\mathscr{R}_{X}(n),\mathscr{R}_{X})\cong\mathscr{H}om_{\mathscr{O}_{X}}(\mathscr{O}_{X}(n),\mathscr{R}_{X})\cong\mathscr{R}_{X}(-n)
$$
$$
\mathscr{R}_{X}(m)\otimes_{\mathscr{R}_{X}}\mathscr{R}_{X}(n)=\mathscr{R}_{X}\otimes_{\mathscr{O}_{X}}\mathscr{O}_{X}(m)\otimes_{\mathscr{R}_{X}}\mathscr{R}_{X}\otimes_{\mathscr{O}_{X}}\mathscr{O}_{X}(n)\cong\mathscr{R}_{X}(m+n)
$$
Therefore, using two propositions above, we have the following computation
\begin{align}\notag
\mathrm{Ext}^{l}_{\mathfrak{g}_{0}\mathfrak{Qco}(X)}(\mathscr{R}_{X}(i),\mathscr{R}_{X}(j))
&=\mathrm{Ext}^{l}_{\mathfrak{g}_{0}\mathfrak{Qco}(X)}(\mathscr{R}_{X},\mathscr{R}_{X}(j)\otimes_{\mathscr{R}_{X}}\mathscr{R}_{X}(-i))\\\notag
&=\mathrm{Ext}^{l}_{\mathfrak{g}_{0}\mathfrak{Qco}(X)}(\mathscr{R}_{X},\mathscr{R}_{X}(j-i))\\\notag
&=H^{l}_{\mathfrak{g}_{0}\mathfrak{Qco}(X)}(X,\mathscr{R}_{X}(j-i))\\\notag
&=H^{l}(X,\mathscr{R}_{X}(j-i))\\\notag
\end{align}
Now we have isomorphism of $\mathscr{O}_{X}$-modules
$$\mathscr{R}_{X}(j-i)=\mathscr{R}_{X}\otimes_{\mathscr{O}_{X}}\mathscr{O}_{X}(i-j)\cong\mathcal{U}(\mathfrak{g}_{0})\otimes_{k}\mathscr{O}_{X}\otimes_{\mathscr{O}_{X}}\mathscr{O}_{X}(i-j)\cong\mathcal{U}(\mathfrak{g}_{0})\otimes_{k}\mathscr{O}_{X}(i-j)$$
Now $\mathcal{U}(\mathfrak{g}_{0})\otimes_{k}\cdot$ is an exact functor, so it commutes with cohomology
$$H^{l}(X,\mathscr{R}_{X}(j-i))\cong H^{l}(X,\mathcal{U}(\mathfrak{g}_{0})\otimes_{k}\mathscr{O}_{X}(i-j))\cong \mathcal{U}(\mathfrak{g}_{0})\otimes_{k}H^{l}(X,\mathscr{O}_{X}(i-j))
$$
Then the result follows from [HR, Ch.III, \S 5].
\end{proof}

This theorem allows us to compute morphisms between twisted sheaves in derived category of $\mathcal{C}h(\mathfrak{g}_{0}\mathfrak{Qco}(X))$. Furthermore, since the bounded complexes subcategory is the localizing subcategory [WC, Ch.X, \S 3, 10.3.15], we are able to compute the morphism sets in derived category of bounded complexes of $\mathfrak{g}_{0}$ quasi-coherent sheaves $\mathcal{D}^{b}(\mathfrak{g}_{0}\mathfrak{Qco}(X))$. More explicitly,
$$\mathrm{Hom}_{\mathcal{D}^{b}(\mathfrak{g}_{0}\mathfrak{Qco}(X))}(\mathscr{O}_{X}(i)[a],\mathscr{O}_{X}(j)[b])=\mathrm{Ext}_{\mathfrak{g}_{0}\mathfrak{Qco}(X)}^{b-a}(\mathscr{O}_{X}(i),\mathscr{O}_{X}(j))
$$

Furthermore, we can freely transfer these morphisms to $\mathcal{D}^{b}(\mathfrak{g}_{0}\mathfrak{Coh}(X))$ by citing a result in [FM, \S 3, 3.5] that also works for the ringed space $(X,\mathscr{R}_{X})$.

\begin{prop}
The natural functor
$$\mathcal{D}^{b}(\mathfrak{g}_{0}\mathfrak{Coh}(X))\rightarrow\mathcal{D}^{b}(\mathfrak{g}_{0}\mathfrak{Qco}(X))
$$
defines a equivalence of categories between $\mathcal{D}^{b}(\mathfrak{g}_{0}\mathfrak{Coh}(X))$ and the full triangulated subcategory $\mathcal{D}^{b}_{\mathfrak{Coh}}(\mathfrak{g}_{0}\mathfrak{Qco}(X))$ of bounded complexes of quasi-coherent sheaves with coherent cohomology.
\end{prop}

\section{Geometry arising from representations for self-commuting $\mathfrak{g}_{1}$}
In this section, we suppose that our superalgebra $\mathfrak{g}=\mathfrak{g}_{0}\oplus\mathfrak{g}_{1}$ satisfies $[\mathfrak{g}_{1},\mathfrak{g}_{1}]=0$. 
\subsection{$\mathbb{Z}$-graded $\mathfrak{g}$-module category}

For the case $[\mathfrak{g}_{1},\mathfrak{g}_{1}]=0$, it's natural to make the following definition: Denote by $\mathcal{M}(\mathfrak{g})$ the left $\mathbb{Z}$-graded $\mathfrak{g}$-module category in which $\mathfrak{g}_{0}$ acts by degree 0 and $\mathfrak{g}_{1}$ acts by degree 1. In other words, a typical object in $\mathcal{M}(\mathfrak{g})$ is of the form $\displaystyle V=\oplus_{i\in\mathbb{Z}}V^{i}$,  in which
$$\mathfrak{g}_{0}V^{i}\subseteq V^{i}, \quad\mathfrak{g}_{1}V^{i}\subseteq V^{i+1}
$$
A simple observation shows that for $V\in \mathrm{Ob}\, \mathcal{M}(\mathfrak{g})$, we have $[\mathfrak{g}_{1},\mathfrak{g}_{1}]V=0$. The set of morphisms between two objects in $\mathcal{M}(\mathfrak{g})$ is given by the degree preserving $\mathfrak{g}$-module homomorphisms. Let $\mathcal{M}^{b}(\mathfrak{g})$ be the full subcategory in $\mathcal{M}(\mathfrak{g})$ consisting of finitely generated(=finite dimensional over k) $\mathfrak{g}$-modules.

One can see that generally, this category might not be very useful (for example, if $[\mathfrak{g}_{1},\mathfrak{g}_{1}]=\mathfrak{g}$, this category is trivial). However, for our special case $[\mathfrak{g}_{1},\mathfrak{g}_{1}]=0$, the constructions above are natural generalization of $\mathbb{Z}$-graded module category of exterior algebra. In other words, in the special case $\mathfrak{g}_{0}$=0, since in that case the universal enveloping algebra of $\mathfrak{g}$, which we denote by $\mathcal{U}(\mathfrak{g})$, is just $\Lambda(\mathfrak{g}_{1})$.

Next, we introduce some operations in $\mathcal{M}(\mathfrak{g})$

\,

\noindent a) Let $V$ be an object in $\mathcal{M}(\mathfrak{g})$. For $m\in\mathbb{Z}$, let
$$V(m)^{i}=V^{i-m},\quad V(m)=\bigoplus_{i\in \mathbb{Z}}V(m)^{i}
$$
with action of $\mathfrak{g}$ on $V(m)$ induced be the action of $\mathfrak{g}$ on $V$. It is clear that the degree shift $V\rightarrow V(m)$ gives a functor $\mathcal{M}(\mathfrak{g})\rightarrow\mathcal{M}(\mathfrak{g})$ which is the identity on morphisms.

\noindent b) For two objects $V, W$ in $\mathcal{M}(\mathfrak{g})$, we define their tensor product over k to be
$$V\otimes W=\bigoplus_{l\in\mathbb{Z}}(V\otimes W)^{l},\quad (V\otimes W)^{l}=\bigoplus_{i+j=l}V^{i}\otimes_{k}W^{j}
$$
and the action is defined by
$$x(v\otimes w)=xv\otimes w+(-1)^{deg\, v}v\otimes xw,\quad x\in \mathfrak{g},\ v\in V, \ w\in W.
$$

\noindent c) Any left graded $\mathfrak{g}$-module $V$ has a canonical structure of a right graded $\mathfrak{g}$-module $V_{r}$ given by
$$vx=(-1)^{deg\,v\,deg\,x}xv
$$
It's clear that the map $V\rightarrow V_{r}$ produces a functor that gives an isomorphism of the categories of left and right graded $\mathfrak{g}$-modules.

\noindent d) Let $V^{\ast}=Hom_{k}(V,k)$ be supplied with the $\mathfrak{g}$-action 
$$(x\phi)(v)=(-1)^{deg\,x\,deg\,v}\phi(xv)
$$
Together with the grading $(V^{\ast})^{i}=\mathrm{Hom}_{k}(V^{-i},k)$, this defines a graded $\mathfrak{g}$-module structure on $V^{\ast}$.

\subsection{The category of rigid complexes that is equivalent to $\mathcal{M}^{b}(\mathfrak{g})$}
Any graded $\mathfrak{g}$-module $V=\oplus_{i\in\mathbb{Z}}V^{i}$ is determined by a family of complexes of $\mathfrak{g}_{0}$-modules indexed by elements $x\in\mathfrak{g}_{1}$
\begin{equation}\notag
L_{x}(V)=\begin{tikzcd}
...\arrow[r] & V^{j-1}\arrow[r,"d^{j-1}(x)"] & V^{j} \arrow[r,"d^{j}(x)"] & V^{j+1}\arrow[r] & ...
\end{tikzcd}
\end{equation}
where $d^{j}(x):V^{j}\rightarrow V^{j+1}$ is the multiplication with $x$, and it's straight forward to check those $d^{j}(x)$ are morphisms of $\mathfrak{g}_{0}$-module. It is clear that $L_{x}(V)$ and $L_{cx}(V)$ for $c\in k^{\ast}$ are isomorphpic, so the the $L_{x}(V)$ are indexed essentially by points of the projectivization $\mathcal{P}(\mathfrak{g}_{1})$ of $\mathfrak{g}_{1}$, which is the self-commuting cone $X$ of $\mathfrak{g}$.

Now using language of $\mathfrak{g}_{0}$-algebraic geometry, we say that a complex of objects in $\mathfrak{g}_{0}\mathfrak{Qco}(X)$ is rigid if it is isomorphic to a complex

$L:$\begin{tikzcd}
...\arrow[r] & V^{j-1}\otimes\mathscr{O}_{X}(j-1)\arrow[r,"d^{j-1}"] & V^{j}\otimes\mathscr{O}_{X}(j) \arrow[r,"d^{j}"] & V^{j+1}\otimes\mathscr{O}_{X}(j+1)\arrow[r] & ...
\end{tikzcd}
in which $d^{j}:V^{j}\otimes\mathscr{O}_{X}(j) \rightarrow V^{j+1}\otimes\mathscr{O}_{X}(j+1)$ are morphisms that commute with $\mathfrak{g}_{0}$ action (in other words, morphisms in $\mathfrak{g}_{0}\mathfrak{Mod}(X)$). A rigid complex is said to be finite if it is bounded and all $V^{j}$ are finite dimensional.

Given a rigid complex $L$, we can construct the graded $\mathfrak{g}$-module $V(L)$ by setting
$$V(L)=\bigoplus_{j\in\mathbb{Z}}V(L)^{j}=\bigoplus_{j\in\mathbb{Z}}\Gamma(L^{j}(-j))
$$
Now $g_{0}$ has a induced action on $\Gamma(L^{j}(-j))$, and the action of elements in $\mathfrak{g}_{1}$ on $V(L)$ is defined by 
$$xv=(-1)^{j}(id\otimes s_{x})\Gamma(d^{j}(-j))(v)
$$
for $x\in\mathfrak{g}_{1}$, $v\in V(L)^{j}$. Here we identify $\Gamma(\mathscr{O}(1))\cong \mathfrak{g}_{1}^{\ast}$ and $s_{x}$ is the evaluation map. Notice that $d^{j}d^{j+1}=0$ ensures that the action defined above is actually an action subalgebra $\mathfrak{g}_{1}\subseteq\mathfrak{g}$. To conclude we can assemble a $\mathfrak{g}$-action, we need to check the following relation:
$$[x,y]v=x(y(v))-y(x(v))
$$
for $x\in\mathfrak{g}_{0},\,y\in\mathfrak{g}_{1}$. By linearity, it suffices to obtain a basis $\{e_{1},e_{2},...,e_{n}\}$ for vector space $\mathfrak{g}_{1}$ and check the relation for $y=e_{k}$, so we make the following computations: denote $a=\Gamma(d^{j}(-j))$, we have a commutative diagram obtained by commutativity of $\mathfrak{g}_{0}$-action and $d^{j}$ (twisted by $\mathscr{O}_{X}(-j)$ and apply global section functor)
\begin{center}
\begin{tikzcd}
V^{j} \arrow[r, "a"] \arrow[d,"x"]
& V^{j+1}\otimes g_{1}^{\ast} \arrow[d, "x"] \\
V^{j} \arrow[r, "a" ]
& V^{j+1}\otimes g_{1}^{\ast}
\end{tikzcd}
\end{center}
we may suppose $a(v)=\sum_{i=1}^{n}\delta_{i}(v)\otimes e_{j}^{\ast}$, then
$$
[x,e_{k}](v)=(-1)^{j}(id\otimes s_{[x,e_{k}]} a(v))
=(-1)^{j}\sum_{i=1}^{n}e_{j}^{\ast}([x,e_{k}])\delta_{i}(v)
$$
$$x(e_{k}(v))=(-1)^{j}x.\delta_{k}(v), \quad e_{k}(x(v))=(-1)^{j}\delta_{k}(xv)
$$
Now $ax=xa$, gives us 
$$\displaystyle\sum_{i=1}^{n}\delta_{i}(xv)\otimes e_{i}^{\ast}=x.\sum_{i=1}^{n}\delta_{i}(v)\otimes e_{i}^{\ast}=\sum_{i=1}^{n}x\delta_{i}(v)\otimes e_{i}^{\ast}+\sum_{i=1}^{n}\delta_{i}(v)\otimes x. e_{i}^{\ast}
$$
Valuing on $e_{k}$, we have
$$\delta_{k}(xv)=x\delta_{k}(v)+\sum_{i=1}^{n}e_{i}^{\ast}([x,e^{k}])\delta_{i}(v)
$$
Multiply $(-1)^{j}$ we get$[x,e_{k}]v=x(e_{k}(v))-e_{k}(x(v))$. Therefore, we conclude that $L(V)$ does give a graded $\mathfrak{g}$-module.

Conversely, a graded $\mathfrak{g}$-module $V$ also gives a rigid complex as follows: the family of $\mathfrak{g}_{0}$-module homomorphisms $d^{j}(x):\,V^{j}\rightarrow V^{j+1}$ which is linear in $x\in\mathfrak{g}_{1}$ is the same as a linear map $V^{j}\rightarrow V^{j+1}\otimes \mathfrak{g}_{1}^{\ast}$, and this map determined a morphism $V^{j}\otimes\mathscr{O}_{X}\rightarrow V^{j+1}\otimes\mathscr{O}_{X}(1)$ on $X=\mathcal{P}(\mathfrak{g}_{1})$, which gives a differential $d^{j}:\,V^{j}\otimes\mathscr{O}_{X}(j)\rightarrow V^{j+1}\otimes\mathscr{O}_{X}(j+1)$ on $X=\mathcal{P}(\mathfrak{g}_{1})$ by tensoring with $\mathscr{O}_{X}(j)$. Now $[\mathfrak{g}_{1},\mathfrak{g}_{1}]=0$ ensures $d^{j+1}d^{j}=0$, and we can also check that these differentials commute with $\mathfrak{g}_{0}$-action: for $x\in\mathfrak{g}_{0}$, it suffices to check the commutativity of the following diagram:
\begin{center}
\begin{tikzcd}
V^{j}\otimes\mathscr{O}_{X}(j)(U) \arrow[r, "d^{j}(U)"] \arrow[d,"x(U)"]
& V^{j+1}\otimes\mathscr{O}_{X}(j+1)(U)  \arrow[d, "x(U)"] \\
V^{j}\otimes\mathscr{O}(j)_{X}(U) \arrow[r, "d^{j}(U)"]
& V^{j+1}\otimes\mathscr{O}_{X}(j+1) (U)
\end{tikzcd}
\end{center}
for any open set $U\subseteq X$ and $x\in\mathfrak{g}_{0}$. Now for some element $v\otimes f$ in $V^{j}\otimes\mathscr{O}_{X}(j)(U)$, where $v\in V^{j}$ and $f$ is a regular function on $U$, the upper horizontal map sends it to $\sum_{i=1}^{n}e_{i}(v)\otimes e_{i}^{\ast}f$, then by the action of $x$, it becomes 
$$x(U) \circ d^{j}(U)\,(v\otimes f)=\sum_{i=1}^{n}x.e_{i}(v)\otimes e_{i}^{\ast}f+\sum_{i=1}^{n}e_{i}(v)\otimes ((x.e_{i}^{\ast})f+e_{i}^{\ast} (x.f))
$$
On the other hand, $x$ send $v\otimes f$ to $x(v)\otimes f+v\otimes x.f$, then it is sent by the bottom arrow to
\begin{align}\notag
d^{j}(U) \circ x(U)\,(v\otimes f)&=\sum_{i=1}^{n}e_{i}(x(v))\otimes e_{i}^{\ast} f+\sum_{i=1}^{n}e_{i}(v)\otimes e_{i}^{\ast} (x.f)\\\notag
&=\sum_{i=1}^{n}([e_{i},x](v)+x.e_{i}(v))\otimes e_{i}^{\ast} f+\sum_{i=1}^{n}e_{i}(v)\otimes e_{i}^{\ast} (x.f)
\end{align}
Now suppose $[x,e_{i}]=\sum_{k=1}^{n}a_{ki}^{x}e_{k}$. Since $(x.e_{i}^{\ast})(e_{k})=-e_{i}^{\ast}([x,e_{k}])=-a_{ik}^{x}$, the action of $x$ on the $\mathfrak{g}_{1}^{\ast}$ is by negative transpose
$x.e_{i}^{\ast}=-\sum_{k=1}^{n}a_{ik}^{x}e_{i}^{\ast}$. Hence we have
$$
\sum_{i=1}^{n}[e_{i},x](v)\otimes e_{i}^{\ast} f=-\sum_{i=1}^{n}\sum_{k=1}^{n}a_{ki}^{x}e_{k}(v)\otimes e_{i}^{\ast} f=\sum_{k=1}^{n}e_{k}(v)\otimes(-\sum_{i=1}^{n}a_{ki}^{x}e_{i}^{\ast} f)=\sum_{k=1}^{n}e_{k}(v)\otimes (x.e_{k}^{\ast}) f
$$
Together with formulas above, we conclude that 
$$
d^{j}(U) \circ x(U)=x(U) \circ d^{j}(U)
$$
so the differentials $d^{j}$ are morphisms commuting with $\mathfrak{g}_{0}$-action. Therefore, from a graded $\mathfrak{g}$-module $V$, we get a rigid complex

\noindent\ $L(V):$\begin{tikzcd}
...\arrow[r] & V^{j-1}\otimes\mathscr{O}_{X}(j-1)\arrow[r,"d^{j-1}"] & V^{j}\otimes\mathscr{O}_{X}(j) \arrow[r,"d^{j}"] & V^{j+1}\otimes\mathscr{O}_{X}(j+1)\arrow[r] & ...
\end{tikzcd}
One can easily see that $V\rightarrow L(V)$ and $L\rightarrow V(L)$
are inverse operations of each other. This gives an one-to-one correspondence between graded $\mathfrak{g}$-module $\mathfrak{g}_{0}$-rigid complexes.

\subsection{Passing to the derived category}

Denote $\mathfrak{g}_{0}\mathrm{Rig}$ the category of rigid complexes in $\mathcal{C}h(\mathfrak{g}_{0}\mathfrak{Qco}(X))$. It's easy to check that the correspondence in the previous section extends to two functors
$$L: \mathcal{M}(\mathfrak{g})\rightarrow\mathcal{C}h^{b}(\mathfrak{g}_{0}\mathfrak{Coh}(X)),\quad V:\mathfrak{g}_{0}\mathrm{Rig}\rightarrow\mathcal{M}(\mathfrak{g})$$
that reveals an equivalence between $\mathcal{M}(\mathfrak{g})$ and $\mathfrak{g}_{0}\mathrm{Rig}$. It's obvious that the finite dimensional graded $\mathfrak{g}$-modules corresponds to rigid complexes with finite length and coherent objects. Moreover, the tensor product is also preserved under this identification

Now composing $L$ with the localization functor $\mathcal{C}h^{b}(\mathfrak{g}_{0}\mathfrak{Coh}(X))\rightarrow\mathcal{D}^{b}(\mathfrak{g}_{0}\mathfrak{Coh}(X))$, we get a functor $\Phi$
$$\Phi:\,\mathcal{M}^{b}(\mathfrak{g})\rightarrow\mathcal{D}^{b}(\mathfrak{g}_{0}\mathfrak{Coh}(X))
$$
In the next section, we will characterize an important feature of $\Phi$.

\section{An induced faithful functor $\tilde{\Phi}$ from stable representation category}

In this section, we continue to suppose $[\mathfrak{g}_{1},\mathfrak{g}_{1}]=0$. Let $\mathcal{P}\subset\mathcal{M}^{b}(\mathfrak{g})$ be the full subcategory of $\mathcal{M}^{b}(\mathfrak{g})$ formed by all projective objects. Two morphisms $f,g:\,V\rightarrow V'$ in $\mathcal{M}^{b}(\mathfrak{g})$ is said to be equivalent if $f-g$ can be factored as $V\rightarrow P\rightarrow V'$ with $P\in\mathrm{Ob}\,\mathcal{P}$, and we write $f\sim g$. It's clear that $\sim$ is a equivalent relation since if $f-g$ factors through $P$, $g-h$ factors through $Q$, then $f-h$ factors through $P\oplus Q$, which is also projective. Now define the category $\mathcal{M}^{b}(\mathfrak{g})/\mathcal{P}$ by
$$
\mathrm{Ob}\,\mathcal{M}^{b}(\mathfrak{g})/\mathcal{P}=\mathrm{Ob}\,\mathcal{M}^{b}(\mathfrak{g}),\quad \mathrm{Hom}_{\mathcal{M}^{b}(\mathfrak{g})/\mathcal{P}}(V, V')=\mathrm{Hom}_{\mathcal{M}^{b}(\mathfrak{g})}(V,V')/\sim
$$
We call this category the stable category of graded $\mathfrak{g}$-modules. 

The main result of this section is the following theorem

\begin{theo}
If $\mathfrak{g}_{0}$ is semisimple, then the functor $\Phi$ constructed above factors through the stable representation category $\mathcal{M}^{b}(\mathfrak{g})/\mathcal{P}$ and induce a faithful functor
$$\tilde{\Phi}: \mathcal{M}^{b}(\mathfrak{g})/\mathcal{P}\rightarrow\mathcal{D}^{b}(\mathfrak{g}_{0}\mathfrak{Coh}(X))
$$
\end{theo}
\noindent We will do this by two steps in the following subsections:

\,

\noindent \textbf{a)} If $V$ is projective, then $\Phi(V)\cong 0$, so any morphism that factors through objects in $\mathcal{P}$ will be sent to zero map by $\Phi$. Therefore, $\Phi$ factors through $\mathcal{M}^{b}(\mathfrak{g})/\mathcal{P}$.

\,

\noindent \textbf{b)} If $\Phi(f)=0$ for some morphism $f$, then $f$ factors through some projective element. Hence $\tilde{\Phi}: \mathcal{M}^{b}(\mathfrak{g})/\mathcal{P}\rightarrow\mathcal{D}^{b}(\mathfrak{g}_{0}\mathfrak{Coh}(X))$ is faithful.
\subsection{Factorization of $\Phi$ through $\mathcal{M}^{b}(\mathfrak{g})/\mathcal{P}$}

Suppose $P$ is an object in $\mathcal{P}$. Denote $\mathcal{F}(\mathfrak{g})$ to be the category of finite-dimensional $\mathfrak{g}$-modules (not neccessarily graded). In [DS, \S 10], we have the following lemma
\begin{lem} 
Suppose $\mathfrak{g}$ is quasi-reductive. If $M\in\mathrm{Ob}\,\mathcal{F}(\mathfrak{g})$ is projective, then its projective associated variety $X_{M}=\varnothing$
\end{lem}

Indeed, the projective objects in $\mathcal{M}^{b}(\mathfrak{g})$ are exactly the objects that are projective in $\mathcal{F}(\mathfrak{g})$, as the following lemma shows:

\begin{lem}
If $\mathfrak{g}_{0}$ is semisimple, then any object $M$ in $\mathcal{P}$ is projective in $\mathcal{F}(\mathfrak{g})$. Conversely, an object in $\mathcal{M}^{b}(\mathfrak{g})$ that is projective in $\mathcal{F}(\mathfrak{g})$ lies in $\mathcal{P}$.
\end{lem}

\begin{proof}
For an object $M$ in $\mathcal{P}$, we have a surjection $\mathcal{U}(\mathfrak{g})\otimes_{\mathcal{U}(\mathfrak{g}_{0})}M\rightarrow M$ defined by $\mathcal{U}(\mathfrak{g})$-module structure. Using PBW theorem, we have an isomorphism of $\mathcal{U}(\mathfrak{g}_{0})$-modules:
$$\mathcal{U}(\mathfrak{g})\cong\mathcal{U}(\mathfrak{g}_{0})\otimes_{k}\Lambda(\mathfrak{g}_{1})
$$
so we have a natural grading of \,$\mathcal{U}(\mathfrak{g})$, which makes it a graded module of itself (More specifically, a monomial $e_{1}e_{2}...e_{i},\,e_{i}\in\mathfrak{g}_{0} \ \text{or}\  \mathfrak{g}_{1}$ in $\mathcal{U}(\mathfrak{g})$ is of degree $d$ if and only if there are $d$ elements of $e_{1},\,e_{2},\,...,\,e_{n}$ in $\mathfrak{g}_{1}$). Now for $x\otimes m$ in $\mathcal{U}(\mathfrak{g})\otimes_{\mathcal{U}(\mathfrak{g}_{0})}M$, we count its degree by $\mathrm{deg}(x)+\mathrm{deg}(m)$. Then one can easily see that the 'evaluation map' $\mathcal{U}(\mathfrak{g})\otimes_{\mathcal{U}(\mathfrak{g}_{0})}M\rightarrow M$ is actually a morphism of graded modules. Then by projectivity of $M$, there exists a splitting map which makes $M$ a direct summand of $\mathcal{U}(\mathfrak{g})\otimes_{\mathcal{U}(\mathfrak{g}_{0})}M\rightarrow M$ as a finite-dimensional graded $\mathfrak{g}$-module, a fortiori as a finite-dimensional $\mathfrak{g}$-module.

Now since $\mathfrak{g}_{0}$ is semisimple, $M$ is projective in the category $\mathcal{F}(\mathfrak{g}_{0})$, the full subcategory of $\mathfrak{Rep}(\mathfrak{g}_{0})$ consists of finite-dimensional representations. Now $\mathcal{U}(\mathfrak{g})\otimes_{\mathcal{U}(\mathfrak{g}_{0})}\cdot$ is a functor from $\mathcal{F}(\mathfrak{g}_{0})$ to $\mathcal{F}(\mathfrak{g})$ which is left adjoint to the forgetful functor from $\mathcal{F}(\mathfrak{g}_{0})$ to $\mathcal{F}(\mathfrak{g})$. More specifically, we have
$$\rm{Hom}_{\mathfrak{g}}(\mathcal{U}(\mathfrak{g})\otimes_{\mathcal{U}(\mathfrak{g}_{0})}M,N)\cong\rm{Hom}_{\mathfrak{g}_{0}}(M,N)
$$
for any $\mathfrak{g}_{0}$-module $M$ and $\mathfrak{g}$-module $N$. Hence $\mathcal{U}(\mathfrak{g})\otimes_{\mathcal{U}(\mathfrak{g}_{0})}\cdot$ is left adjoint to an exact functor, which indicates that it preserves projectives. So $\mathcal{U}(\mathfrak{g})\otimes_{\mathcal{U}(\mathfrak{g}_{0})}M$ is a projective object in $\mathcal{F}(\mathfrak{g})$.
Now $M$ is its summand, so $M$ is also projective in $\mathcal{F}(\mathfrak{g})$.

Conversely, suppose $M$ is projective in the category $\mathcal{F}(\mathfrak{g})$. Given any surjection of graded $\mathfrak{g}$-modules $f:
\,N\rightarrow L$ and a graded morphism $g:\,M\rightarrow L$, by projectivity of $M$, we have a lifting $h:\,M\rightarrow N$ in $\mathcal{F}(\mathfrak{g})$, not necessarily graded. Now suppose $p_{i}$ is the projection $M\rightarrow M^{i}$, it is obvious a $\mathfrak{g}_{0}$-module homomorphism. Now define a graded map $h':\,M\rightarrow N$ by
$$ h'(\sum_{i\in\mathbb{Z}}x_{i})=\sum_{i\in\mathbb{Z}}p_{i}(h(x_{i})), \ x_{i}\in M^{i}
$$
It is a $\mathfrak{g}_{0}$-module homomorphism since $h$ and all $p^{i}$'s are. Now for any $e\in\mathfrak{g}_{1}$, we have
$$
h'(e\sum_{i\in\mathbb{Z}}x_{i})=\sum_{i\in\mathbb{Z}}p^{i+1}(h(ex_{i}))
=\sum_{i\in\mathbb{Z}}p^{i+1}(eh(x_{i}))\\\notag
=\sum_{i\in\mathbb{Z}}ep^{i}(h(x_{i}))
=eh'(\sum_{i\in\mathbb{Z}}x_{i})
$$
in which the second equality is because h is a $\mathfrak{g}$-module homomorphism, and the third is by a simple observation that $p^{i+1}(ey)=ep^{i}(y)$ for $y\in M^{i}$. Furthermore, since $f$ is a lifting, we have $fh=g$, so
$g(x_{i})=fh(x_{i})=f(h(x_{i}))$. Now since $f,\,g$ are graded maps and
$\displaystyle f(h(x_{i}))=\sum_{i\in\mathbb{Z}}f(p_{i}(h(x_{i})))\in L^{i}$, we see that $f(h(x_{i}))=f(p_{i}(h(x_{i})))$, Hence
$$fh'(\sum_{i\in\mathbb{Z}}x_{i})=\sum_{i\in\mathbb{Z}}f(p_{i}(h(x_{i})))=\sum_{i\in\mathbb{Z}}f((h(x_{i}))=fh(\sum_{i\in\mathbb{Z}}x_{i})=g(\sum_{i\in\mathbb{Z}}x_{i})
$$
so $h'$ is the required graded lifting. Then $M$ is in $\mathcal{P}$, completing the proof.
\end{proof}

Now for our projective graded module $P$, we cite the constructions of the localization of the DS functor in [DS]: Consider the trivial vector bundle (as an $\mathscr{O}_{X}$- module, temporarily forget the $\mathfrak{g}_{0}$-module structure) $\mathscr{O}_{X}\otimes P$  with fiber isomorphic to $P$. Let $\partial:\,\mathscr{O}_{X}\otimes P\rightarrow\mathscr{O}_{X}\otimes P$ be the map defined by
$$\partial\phi(x)=x\phi(x)
$$
for $x\in X=\mathcal{P}(\mathfrak{g}_{1}),\,\phi\in\mathscr{O}_{X}\otimes P$. Clearly $\partial^{2}=0$ and the cohomology $\mathcal{H}(P)=\rm{ker}\partial/\rm{im}\partial$ is coherent. Furthermore, one can easily check the relation
$$\mathcal{H}(P)=\bigoplus_{i\in\mathbb{Z}}H^{i}(L(P))
$$
in which $L(P)$ is the corresponding rigid complex for $P$. Now [DS, \S 11, 11.2 (1)] gives us the following lemma,
\begin{lem}
Let $M$ be a finite-dimensional $\mathfrak{g}$-module, then the support of $\mathcal{H}(M)$ is contained in $X_{M}$.
\end{lem}

Now we can easily proof the our statement: For an object in $\mathcal{P}$, by lemma 5.2, it is also projective in $\mathcal{F}(\mathfrak{g})$. Then by lemma 5.1, its projective associated variety $X_{M}=\varnothing$. Hence according to lemma 5.3, the sheaf $\mathcal{H}(P)=0$. Then by remark above, $\bigoplus_{i\in\mathbb{Z}}H^{i}(L(P))=\mathcal{H}(P)=0$, so $H^{i}(L(P))$ for all $i\in\mathbb{Z}$. This means $L(P)$ is acyclic, hence quasi-isomorphic to zero complex. Therefore any morphism factoring through such a projective object will be sent to 0 by $\Phi$, which means $\Phi$ factors through $\mathcal{M}^{b}(\mathfrak{g})/\mathcal{P}$, inducing a functor $\tilde{\Phi}:\mathcal{M}^{b}(\mathfrak{g})/\mathcal{P}\rightarrow\mathcal{D}^{b}(\mathfrak{g}_{0}\mathfrak{Coh}(X))$.

\subsection{The kernel of $\Phi$ on morphisms}
Consider the composition of functors $\Psi:=\mathscr{R}_{X}\otimes_{\mathscr{O}_{X}}F(\cdot)$, where $F$ is the forgetful functor $\mathfrak{g}_{0}\mathfrak{Coh}(X)\rightarrow\mathfrak{Qco}(X)$. One see that $\Psi$ is an exact functor from $\mathfrak{g}_{0}\mathfrak{Coh}(X)$ to itself, so it naturally extend to a derived functor $\tilde{\Psi}:\mathcal{D}^{b}(\mathfrak{g}_{0}\mathfrak{Coh}(X))\rightarrow\mathcal{D}^{b}(\mathfrak{g}_{0}\mathfrak{Coh}(X))$ with the following commutative diagram:
\begin{center}
\begin{tikzcd}
\mathcal{M}^{b}(\mathfrak{g})\arrow[d,"q'"]\arrow[dr,"\Phi"]\arrow[r,"L"] &\mathfrak{g}_{0}\mathfrak{Coh}(X)\arrow[r,"\Psi"]\arrow[d,"q"] & \mathfrak{g}_{0}\mathfrak{Coh}(X)\arrow[d,"q"]\\ 
\mathcal{M}^{b}(\mathfrak{g})/\mathcal{P}\arrow[r,"\tilde{\Phi}"] &\mathcal{D}^{b}(\mathfrak{g}_{0}\mathfrak{Coh}(X))\arrow[r,"\tilde{\Psi}"] & \mathcal{D}^{b}(\mathfrak{g}_{0}\mathfrak{Coh}(X))
\end{tikzcd}
\end{center}
To prove the faithfulness of $\tilde{\Phi}$, it suffice to show that the composition $\tilde{\Psi}\circ\tilde{\Phi}$ is faithful.

For a morphism $\varphi: V\rightarrow W$ in $\mathcal{M}^{b}(\mathfrak{g})$,  suppose $\tilde{\varphi}=\tilde{\Psi}\circ\Phi(\varphi)=0$, we firstly show that this implies $\varphi=0$ if we impose an additional condition on the $\mathfrak{g}$-module $W$, namely $\Lambda^{n}(\mathfrak{g}_{1})W=0$. Such $W$'s will be called \emph{reduced} modules.

For convenience, we write $\mathcal{D}^{b}=\mathcal{D}^{b}(\mathfrak{g}_{0}\mathfrak{Coh}(X))$ and $R=\Psi\circ L$, then $\tilde{\Psi}\circ\Phi$ is actually $q\circ R$. Now suppose $i$ is the largest integer with $V^{i}\oplus W^{i}\neq 0$, then the morphism $\varphi$ gives a morphism of two distinguished triangle:
\begin{center}
\begin{tikzcd}
R(V)^{i}[-i]\arrow[r]\arrow[d,"\tilde{\varphi}^{i}"] & R(V)\arrow[r]\arrow[d,"\tilde{\varphi}"] & R(V)'\arrow[d,"\tilde{\varphi'}"]\\ 
R(W)^{i}[-i]\arrow[r] & R(W)\arrow[r] & R(W)'
\end{tikzcd}
\end{center}
in which $R(V)^{i}$, $R(W)^{i}$ are complexes with the $i^{th}$ term of $R(V)$, $R(W)$ concentrated on degree 0. $R(V)'$ and $R(W)'$ are defined by truncation: $\sigma_{<i}R(V)$ and $\sigma_{<i}R(W)$. $\tilde{\varphi}^{i}$ and $\varphi'$ are the morphisms in derived category induced by the restrictions of $\varphi$ on $V^{i}$ and $V/V^{i}$. If there exists a different morphism $\phi$ such that $\tilde{\Psi}\circ\Phi(\phi)=\tilde{\Psi}\circ\Phi(\varphi)$, then $\phi$ also gives such a morphism of triangles with the middle morphism coincides. In other words, both extensions $\varphi^{i}$ and $\phi^{i}$ makes the left part of the diagram commute. Then by the long exact sequence induced by $\mathrm{Hom}(R(V)^{i}[-i],\cdot)$ 
(See [GM, Ch.IV, \S 1, 1.3]), we see that $\tilde{\varphi^{i}}$ and $\tilde{\phi^{i}}$ differ by an element that can be factored through $R(V)^{i}[-i]\rightarrow R(W)'[-1]\rightarrow R(W)^{i}[-i]$.

Let us show that if $W$ is reduced, then $\mathrm{Hom}_{\mathcal{D}^{b}}(R(V)^{i}[-i], R(W)'[-1])$ acts trivially on the set of extensions $R(V)^{i}[-i]\rightarrow R(W)^{i}[-i]$. Suppose $m$ is the largest integer such that $W^{m}\neq 0$, then $m\leq i$. Consider the triangle:
$$R(W)^{m}[-m-1]\rightarrow R(W)'[-1]\rightarrow R(W)''[-1]
$$
in which $R(W)''=\sigma_{<m}R(W)'$, we firstly notice that, 
$$\mathrm{Hom}_{\mathcal{D}^{b}}(R(V)^{i}[-i],R(W)^{m}[-m-1])=\mathrm{Ext}^{i-m-1}_{\mathfrak{g}_{0}\mathfrak{Qco}(X)}(\mathscr{R}_{X}(i),\mathscr{R}_{X}(m))\otimes_{k}\mathrm{Hom}_{k}(V^{i},W^{m})
$$
Since $m<i$, by Theorem 3.1, This is nonzero if and only if $i=m+n$. Splitting off the components of $R(W)'[-1]$ one by one (starting from the right) as above, we see that $\mathrm{Hom}_{\mathcal{D}^{b}}(R(V)^{i}[-i], R(W)'[-1])$ is a quotient of $\mathcal{U}(\mathfrak{g}_{0})\otimes_{k}\mathrm{Hom}_{k}(V^{i}, W^{i-n}\otimes (\Lambda^{n}\mathfrak{g}_{1}))$, and an elemnt of this latter Hom yields the morphism $V^{i}\rightarrow W^{i}$ equal to its composition with the $\Lambda(\mathfrak{g}_{1})$ multiplication induced by $\mathfrak{g}$-module structure: $W^{i-n}\otimes \Lambda^{n}(\mathfrak{g}_{1})$, and this determines the extension $R(V)^{i}[-i]\rightarrow R(W)^{i}[-i]$. If $W$ is reduced, one see that the multiplication $W^{i-n}\otimes \Lambda^{n}(\mathfrak{g}_{1})$ is zero, so $\tilde{\varphi}^{i}$ is determined uniquely. In particular, if $\tilde{\varphi}=0$, then $R(\varphi^{i}): R(V)^{i}\rightarrow R(W)^{i}$ is a zero map (for complexes concentrated in zero degree, the morphism set in derived category coincides with the morphisms between complexes).

But if $W$ is reduced, then $W'=W/W^{m}$ is also reduced, so by induction, we see that $R(\varphi): R(V)\rightarrow R(W)$ is a zero map. Now $R(\varphi)=\mathscr{R}_{X}\otimes_{\mathscr{O}_{X}}F(\varphi)=\mathcal{U}(\mathfrak{g}_{0})\otimes_{k}L(\varphi)$ as morphisms between sheaves of abelian groups, so this would also implies $L(\varphi)=0$. However, $L$ is a equivalence of categories, so $\varphi=0$.

Now all we need is the following lemma
\begin{lem}
Suppose $V$ is a finite-dimensional 
 $\mathfrak{g}$-module $\mathrm{(}[\mathfrak{g}_{1},\mathfrak{g}_{1}]=0\mathrm{)}$. Now if $\mathfrak{g}_{0}$ is semisimple, then there exist a projective object $P$ in $\mathcal{F}(\mathfrak{g})$ and a reduced module $M$ such that $V\cong P\oplus M$ as $\mathfrak{g}$-modules.
\end{lem}
\begin{proof}
Suppose $e_{1},\,e_{2},\,...,\,e_{n}$ is a vector space basis for $\mathfrak{g}_{1}$, then one can easily see that the set
$$K=\{v\in V|e_{1}e_{2}...e_{n}v=0\}
$$
is a $\mathfrak{g}$-submodule of V. Since $\mathfrak{g}_{0}$ is semisimple, by Weyl's theorem, there exist a $\mathfrak{g}_{0}$-submodule  $Q\subseteq V$ such that  $Q\cap K=\varnothing$, $V=Q\oplus K$ as $\mathfrak{g}_{0}$-modules. 
Now consider the induced module $\mathcal{U}(\mathfrak{g})\otimes_{\mathcal{U}(\mathfrak{g}_{0})} Q$, we have a $\mathfrak{g}$-module homomorphism $\mathcal{U}(\mathfrak{g})\otimes_{\mathcal{U}(\mathfrak{g}_{0})} Q\rightarrow V$ defined by $x\otimes q\rightarrow x.q$. This map is indeed an injection: If $\sum_{i=1}^{m}\lambda_{i}\otimes q_{i}$ is sent to zero, in which $q_{i}\in Q,\,q_{i}\neq 0$ and $\lambda_{i}$'s are different monomials of the form $e_{i_{1}}e_{i_{2}}...e_{i_{l}}$ in $\Lambda(\mathfrak{g}_{1})$, $e_{i}\in \mathfrak{g}_{1}$. we have  $\sum_{i=1}^{m}\lambda_{i}.q_{i}=0$. Without loss of generality, suppose $\lambda_{1}=e_{1}e_{2}...e_{d}$ is of the minimal degree $d$, by multiplying $e_{d+1}e_{d+2}...e_{n}$, we see that $e_{1}e_{2}...e_{n}.q_{1}=0$, then $q_{1}\in K$. Since $q_{1}\neq 0$, this contradicts to the fact that $Q\cap K=\varnothing$. Hence we can identify $\mathcal{U}(\mathfrak{g})\otimes_{\mathcal{U}(\mathfrak{g})} Q$ as a submodule of $V$.

Now there is a isomorphism of $\mathfrak{g}$-modules
$$\mathcal{U}(\mathfrak{g})\otimes_{\mathcal{U}(\mathfrak{g}_{0})} Q\cong\mathrm{Hom}_{\,\mathcal{U}(\mathfrak{g}_{0})}(\mathcal{U}(\mathfrak{g}),\Lambda^{n}(\mathfrak{g}_{1})\otimes Q)
$$
Let we describe this. For $m\in \mathcal{U}(\mathfrak{g})\otimes_{\mathcal{U}(\mathfrak{g}_{0})} Q$, we define a morphism $f(m):\mathcal{U}(\mathfrak{g})\rightarrow\Lambda^{n}(\mathfrak{g}_{1})\otimes Q$ by $f(m)(e)=(em)^{n}$, here $(em)^{n}$ is the image of $em$ by the $\mathfrak{g}_{0}$-projection to the $n^{th}$ component: $\mathcal{U}(\mathfrak{g})\otimes_{\mathcal{U}(\mathfrak{g}_{0})} Q\rightarrow\Lambda^{n}(\mathfrak{g}_{1})\otimes Q$. 

Conversely, for $\beta\in\mathrm{Hom}_{\,\mathcal{U}(\mathfrak{g}_{0})}(\mathcal{U}(\mathfrak{g}),\Lambda^{n}(\mathfrak{g}_{1})\otimes Q)$, we define an element $g(\beta)=\sum\beta(\lambda_{i})\bar{\lambda_{i}}$. Here $\lambda_{i}$'s are all monomials in $e_{1},e_{2},...e_{n}$ and for $\lambda_{i}=e_{i_{1}}e_{i_{2}}...e_{i_{l}}$, we define $\bar{\lambda_{i}}$ to be $e_{i_{1}'}e_{i_{2}'}...e_{i_{n-l}'}$ in which $\{i_{1}',i_{2}',...,i_{n-l}'\}$ is the complement of $\{i_{1},i_{2},...,i_{l}\}$. It's straight to check $f$ is a $\mathfrak{g}$-module homomorphism and $fg=gf=1$, so $f$ is actually an isomorphism, which gives the relation above.

Since the functor $\mathrm{Hom}_{\,\mathcal{U}(\mathfrak{g}_{0})}(\mathcal{U}(\mathfrak{g}),\cdot)$ is right adjoint to the forgetful functor, it preserves injectives. Since $\mathfrak{g}_{0}$ is semisimple, we see that $\mathcal{U}(\mathfrak{g})\otimes_{\mathcal{U}(\mathfrak{g}_{0})} Q$ is injective in $\mathcal{F}(\mathfrak{g})$. Then the injection $\mathcal{U}(\mathfrak{g})\otimes_{\mathcal{U}(\mathfrak{g})} Q\rightarrow V$ splits as $V\cong(\mathcal{U}(\mathfrak{g})\otimes_{\mathcal{U}(\mathfrak{g})} Q)\oplus M$ for some submodule $M$. Now since $\mathcal{U}(\mathfrak{g})\otimes_{\mathcal{U}(\mathfrak{g})} \cdot$ is left adjoint to forgetful functor, it preserves projectives, hence $\mathcal{U}(\mathfrak{g})\otimes_{\mathcal{U}(\mathfrak{g})}$ is projective in $\mathcal{F}(\mathfrak{g})$. Finally, $dim_{k}\,\Lambda^{n}(\mathfrak{g}_{1})\otimes Q=dim_{K}\,Q=dim_{k}\Lambda^{n}(\mathfrak{g}_{1})V$ by definition of $K$ and $V=Q\oplus K$, we see that $\Lambda^{n}(\mathfrak{g}_{1})M=0$, hence $M$ is reduced, completing the proof.
\end{proof}

Now with this lemma, we can complete the proof. Let $\varphi:V\rightarrow W,\, W=P\oplus M$ with $P$ free, $M$ reduced, and let $\tilde{\Psi}\circ\Phi(\varphi)=0$. By arguments in the beginning of this subsection, the composition $V\rightarrow W\rightarrow M$ is zero. Therefore, $\varphi$ is of the form $V\rightarrow P\rightarrow W$, so $\varphi=0$ in the stable category, completing the proof.

\section{Additional Topics}

In these two sections, we give a reason for why $\tilde{\Phi}$ is not full and generalize our method to an arbitrary lie superalgebra.

\subsection{Why functor $\tilde{\Phi}$ is not full}
In [GM, Ch.IV, \S 3.9], it is shown that the functor $\tilde{\Phi}$ is full when $\mathfrak{g}_{0}=0$. However, in general case, We do have a reason to believe that the fullness of this functor fail due to the lack of semisimplicity of the \,$\mathcal{U}(\mathfrak{g}_{0})$-module (not necessarily finite dimensional) category. 

Indeed, if $\tilde{\Phi}$ is full, given two modules $V$ and $W$ that concentrate on different degree, Since there are no nonzero morphism between $V$ and $W$ in $\mathcal{M}^{b}(\mathfrak{g})$, we should have $\mathrm{Hom}_{\mathcal{D}^{b}}(\Phi(V),\Phi(W))=0$. Now suppose $V=V^{i}$, $W=W^{j}$, $i\neq j$, we will have:
\begin{align}\notag
\mathrm{Hom}_{\mathcal{D}^{b}}(\Phi(V),\Phi(W))&=\mathrm{Hom}_{\mathcal{D}^{b}}(\mathscr{O}_{X}(i)\otimes_{k}V[-i],\mathscr{O}_{X}(j)\otimes_{k}W[-j])\\\notag
&=\mathrm{Ext}_{\mathscr{R}_{X}}^{i-j}(\mathscr{O}_{X}(i)\otimes_{k}V,\mathscr{O}_{X}(j)\otimes_{k}W)
\end{align}
Now there is a canonical isomorphism of $\mathscr{R}_{X}$-modules:
$$\mathscr{O}_{X}\otimes_{k}V=\mathscr{O}_{X}\otimes_{k}\overline{\mathcal{U}(\mathfrak{g}_{0})}\otimes_{\overline{\mathcal{U}(\mathfrak{g}_{0})}}V=\mathscr{R}_{X}\otimes_{\overline{\mathcal{U}(\mathfrak{g}_{0})}}V
$$
in which \,$\overline{\mathcal{U}(\mathfrak{g}_{0})}$ is the constant sheaf of rings \,$\mathcal{U}(\mathfrak{g}_{0})$. To compute the cohomology, we notice a pair of adjoint functors:

$$\mathrm{Hom}_{\mathscr{R}_{X}}(\mathscr{R}_{X}\otimes_{\overline{\mathcal{U}(\mathfrak{g}_{0})}}\mathscr{F},\mathscr{G})=\mathrm{Hom}_{\overline{\mathcal{U}(\mathfrak{g}_{0})}}(\mathscr{F},\mathscr{G})
$$

Since the functor $\mathscr{R}_{X}\otimes_{\overline{\mathcal{U}(\mathfrak{g}_{0})}}$ can be identified as $\mathscr{O}_{X}\otimes_{k}$ as sheaves of abelian groups, it is exact. Hence, the forgetful functor from $\mathscr{R}_{X}$-module to $\overline{\mathcal{U}(\mathfrak{g}_{0})}$-module is right adjoint to a exact functor, then it preserves injectives. On the other hand, the forgetful functor is exact itself, so the adjointness induces relations of derived functors:
$$\mathrm{Ext}^{k}_{\mathscr{R}_{X}}(\mathscr{R}_{X}\otimes_{\overline{\mathcal{U}(\mathfrak{g}_{0})}}\mathscr{F},\mathscr{G})=\mathrm{Ext}^{k}_{\overline{\mathcal{U}(\mathfrak{g}_{0})}}(\mathscr{F},\mathscr{G})
$$
Now if  $\mathfrak{F}$ is the constant sheaf of $\mathfrak{g}_{0}$-modules $V$, we can explicitly compute this cohomology, as follows: the functors $\mathrm{Ext}^{k}_{\overline{\mathcal{U}(\mathfrak{g}_{0})}}(V,\cdot)$ is the derived functor of $\mathrm{Hom}_{\overline{\mathcal{U}(\mathfrak{g}_{0})}}(V,\cdot)$. Now another pair of adjoint functor tells us
$$\mathrm{Hom}_{\overline{\mathcal{U}(\mathfrak{g}_{0})}}(V,\mathscr{G})=\mathrm{Hom}_{\mathcal{U}(\mathfrak{g}_{0})}(V,\Gamma(X,\mathscr{G}))
$$
Hence the functor $\mathrm{Hom}_{\overline{\mathcal{U}(\mathfrak{g}_{0})}}(V,\cdot)$ is actually the composition of the global section functor and the $\mathrm{Hom}$ functor $\mathrm{Hom}_{\mathcal{U}(\mathfrak{g}_{0})}(V,\cdot)$ in the \,$\mathcal{U}(\mathfrak{g}_{0})$-module category. The above adjointness reveals that the global section functor is right adjoint to the exact constant sheaf functor, so it preserves injectives. By [WC, \S 5.8, 5.8.2], we have a convergent Grothendieck spectral sequence:
$$^{I}E_{2}^{pq}=Ext^{p}_{\mathcal{U}(\mathfrak{g}_{0})}(V,H^{q}(\mathscr{G}))\Rightarrow Ext_{\overline{\mathcal{U}(\mathfrak{g}_{0})}}^{p+q}(V,\mathscr{G})
$$

Summarizing above, when $i>j$, Note that the cohomology of $\mathscr{O}_{X}(j-i)$ vanishes for all degree but $n-1$, we have the following computations:
\begin{align}\notag
\mathrm{Ext}_{\mathscr{R}_{X}}^{i-j}(\mathscr{O}_{X}(i)\otimes_{k}V,\mathscr{O}_{X}(j)\otimes_{k}W)&=\mathrm{Ext}_{\mathscr{R}_{X}}^{i-j}(\mathscr{R}_{X}(i)\otimes_{\overline{\mathcal{U}(\mathfrak{g}_{0})}}V,\mathscr{R}_{X}(j)\otimes_{\overline{\mathcal{U}(\mathfrak{g}_{0})}}W)\\\notag
&=\mathrm{Ext}_{\mathscr{R}_{X}}^{i-j}(\mathscr{R}_{X}\otimes_{\overline{\mathcal{U}(\mathfrak{g}_{0})}}V,\mathscr{R}_{X}(j-i)\otimes_{\overline{\mathcal{U}(\mathfrak{g}_{0})}}W)\\\notag
&=\mathrm{Ext}_{\overline{\mathcal{U}(\mathfrak{g}_{0})}}^{i-j}(V,\mathscr{R}_{X}(j-i)\otimes_{\overline{\mathcal{U}(\mathfrak{g}_{0})}}W)\\\notag
&=\mathrm{Ext}_{\overline{\mathcal{U}(\mathfrak{g}_{0})}}^{i-j}(V,\mathscr{O}_{X}(j-i)\otimes_{k}W)\\\notag
&=\mathrm{Ext}_{{\mathcal{U}(\mathfrak{g}_{0})}}^{i-j-n+1}(V,H^{n-1}(\mathscr{O}_{X}(j-i)\otimes_{k}W))\\\notag
&=\mathrm{Ext}_{{\mathcal{U}(\mathfrak{g}_{0})}}^{i-j-n+1}(V,S^{i-j-n}(\mathfrak{g}_{1})\otimes_{k}W))
\end{align}

Now we can easily come up with a example where this cohomology does not vanish. Let $\mathfrak{g}_{0}=\mathfrak{sl}(2,k)$ and $\mathfrak{g}_{1},\,V,\,W$ be trivial $\mathfrak{g}_{0}$-modules. By [WC, \S 7.7, Exercise 7.7.3], $H^{3}(\mathfrak{sl}(2,k),k)=k$, Hence:
\begin{align}\notag
\mathrm{Ext}_{{\mathcal{U}(\mathfrak{g}_{0})}}^{i-j-n+1}(V,S^{i-j-n}(\mathfrak{g}_{1})\otimes_{k}W))&=V\otimes_{k}S^{i-j-n}(\mathfrak{g}_{1})\otimes_{k}W\otimes_{k}H^{3}(\mathfrak{sl}(2,k),k)\\\notag
&=V\otimes_{k}S^{i-j-n}(\mathfrak{g}_{1})\otimes_{k}W
\end{align}
which does not vanish. 

\subsection{$\mathfrak{g}_{0}$-geometry in nongraded case}
For a general lie superalgebra $\mathfrak{g}=\mathfrak{g}_{0}\oplus\mathfrak{g}_{1}$ and a $\mathfrak{g}$-module V, there is a morphism $V\rightarrow V\otimes_{k}\mathfrak{g}_{1}^{\ast}$ induced by the $\mathfrak{g}_{1}$-action, hence gives a morphism between sheaves:
$d:\,\mathscr{O}_{\mathcal{P}(\mathfrak{g}_{1})}\otimes V\rightarrow\mathscr{O}_{\mathcal{P}(\mathfrak{g}_{1})}(1)\otimes V$. Restricting to the projective self-commuting cone, we will get a complex:
\begin{center}
\begin{tikzcd}
...\arrow[r] & \mathscr{O}_{X}(i-1)\otimes V \arrow[r,"d(i-1)"] & \mathscr{O}_{X}(i)\otimes V \arrow[r,"d(i)"] & \mathscr{O}_{X}(i+1)\otimes V\arrow[r] & ... 
\end{tikzcd}
\end{center}

Actually, this gives a functor $\mathcal{F}(\mathfrak{g})\rightarrow\mathcal{C}h(\mathfrak{g}_{0}\mathfrak{Coh}(X))$. When compositing with localization functor $\mathcal{C}h(\mathfrak{g}_{0}\mathfrak{Coh}(X))\rightarrow\mathcal{D}^{b}(\mathfrak{g}_{0}\mathfrak{Coh}(X))$, we see that if $\mathfrak{g}$ is quasi-reductive, it still factors through the stable representation category: For a projective object $P$ in $\mathcal{F}(\mathfrak{g})$, by lemma 5.1, the projective associated variety $X_{M}=\varnothing$. Then by lemma 5.3, the cohomology of the complex is supported in $X_{M}=\varnothing$, hence vanishes everywhere. Then the complex is acyclic, hence quasi-isomorphic to $0$. Hence any morphism factoring through projective objects will be sent to zero.

The reason why faithfulness may fail in this case is that the sheaf cohomology for $\mathfrak{R}_{X}$ or more specifically for $\mathfrak{O}_{X}$ doesn't have a satisfying vanishing theorem, so the generally we don't have such an elegant way to compute morphisms in derived category.

This complex is sometimes really useful when restricted to some subvariety. For example, as [DS, \S 11, 11.3] mentioned, for $\mathfrak{g}=\mathfrak{gl}(m|n)$, write $\mathfrak{g}=\mathfrak{g}^{-1}\oplus\mathfrak{g}^{0}\oplus\mathfrak{g}^{1}$ in which $\mathfrak{g}^{1}$ and $\mathfrak{g}^{-1}$ are the upper and lower quasi-triangular matrices. Then the complex of global sections when restricted to $\mathcal{P}(\mathfrak{g}^{1})$
$$
0\rightarrow  V \rightarrow (\mathfrak{g}^{1})^{\ast}\otimes V \rightarrow S^{2}(\mathfrak{g}^{1})^{\ast}\otimes V\rightarrow ...\rightarrow S^{p}(\mathfrak{g}^{1})^{\ast}\otimes V\rightarrow ...
$$
is nothing else but the Koszul complex computing the cohomology $H^{\cdot}(\mathfrak{g}_{1},V)$, which plays an important role in the Kazhdan-Luszting theory for $\mathcal{F}(GL(m|n))$, as [DS, \S 11, 11.3] remarked.


\begin{thebibliography}{MMM,10}
			
\bibitem[DS]{DS}Gorelik, M., Hoyt, C., Serganova, V. et al. \emph{The Duflo–Serganova Functor, Vingt Ans Après}. J Indian Inst Sci \textbf{102}, 961–1000 (2022). https://doi.org/10.1007/s41745-022-00334-9

\bibitem[FM]{FM}Huybrechts, D., \emph{Fourier-Mukai Transforms in Algebraic Geometry} (Oxford, 2006; online edn, Oxford Academic, 1 Sept. 2007), https://doi.org/10.1093/acprof:oso/9780199296866.001.0001, accessed 6 Sept. 2023.

\bibitem[GM]{GM}Sergei Gelfand, Yuri Manin:
\emph{Methods of homological algebra},
transl. from the 1988 Russian (Nauka Publ.) original,
Springer 1996. xviii+372 pp. 2nd corrected ed. 2002
doi:10.1007/978-3-662-12492-5

\bibitem[HR]{HR}Hartshorne, Robin. \emph{Algebraic Geometry}. Vol. 52. : Springer, 1977.

\bibitem[WC]{WC}Weibel, Charles A.. \emph{An introduction to homological algebra}. Vol. 38. : Cambridge University Press, Cambridge, 1994.
\end{thebibliography}
\end{document}